\documentclass[english]{ourlematema}

\usepackage{color}

\newcommand{\PP}{\mathbb{P}}
\newcommand{\SSS}{\mathbb{S}}
\newcommand{\cL}{\mathcal{L}}
\newcommand{\NM}{\mathrm{NM}}
\newcommand{\Bad}{\mathrm{Bad}}
\newcommand{\Gr}{\mathrm{Gr}}

\DeclareMathOperator{\adj}{adj}
\DeclareMathOperator{\rank}{rank}
\DeclareMathOperator{\tr}{trace}

\title{Linear Spaces of Symmetric Matrices with Non-Maximal Maximum Likelihood Degree}
\titlemark{LSSMs with Non-Maximal ML-Degree}
\author{Yuhan Jiang}
\address{Harvard University \\ email: yjiang@math.harvard.edu}
\author{Kathl\'{e}n Kohn}
\address{KTH Royal Institute of Technology, Stockholm \\ email: kathlen@kth.se}
\author{Rosa Winter}
\address{MPI for Mathematics in the Sciences, Leipzig\\ email: rosa.winter@mis.mpg.de}

\begin{document}
\begin{abstract}
    We study the maximum likelihood degree of linear concentration models in algebraic statistics. We relate the geometry of the reciprocal variety to that of semidefinite programming. We show that the Zariski closure in the Grassmanian of the set of linear spaces that do not attain their maximal possible maximum likelihood degree coincides with the Zariski closure of the set of linear spaces  defining a projection with non-closed image of the positive semidefinite cone. In particular, this shows that this closure is a union of coisotropic hypersurfaces.
\end{abstract}
\section{Introduction}

Maximum likelihood estimation is a widespread optimization approach to fit empirical data to a statistical model. 
The \emph{maximum likelihood degree} (or short, \emph{ML-degree}) of a model is the number of complex critical points of this optimization problem for generic empirical data~\cite{catanese2006maximum}.
The aim of this paper is to study models whose actual maximum likelihood degree differs from the expected one.

The models we consider are sets of multivariate Gaussian distributions with mean zero that are linear in the space of concentration matrices. 
The concentration matrices of such a model form a spectrahedron that is the intersection of a linear subspace $\cL$ of the space $\SSS^n$ of symmetric $n \times n$ matrices with the cone of positive definite matrices. 

The ML-degree of a model defined by a generic subspace $\cL \subset \SSS^n$ is known to be the degree of its \emph{reciprocal variety} that is parametrized by the inverses of all matrices in $\cL$ (see \cite[Theorem 2.3]{SU10}).
We say that a linear subspace $\cL \subset \SSS^n$
with that property is \emph{ML-maximal}.

\subsection{Main results}
We fix the bilinear pairing
$(X,Y)\mapsto  \tr(XY)$
on the vector space $\SSS^n$ of complex symmetric $n \times n$ matrices.
This is an inner product when
restricting to the real symmetric matrices.
For a linear subspace $\cL \subset \SSS^n$,
we write 
$$ \cL^\perp := \{ Y \in \SSS^n \mid \tr(XY)=0 \text{ for all } X \in \cL \} $$
for its \emph{annihilator} or \emph{polar linear space} with respect to the trace pairing.
Moreover, we consider the Zariski closure 
$$\cL^{-1} := \overline{\{ X^{-1} \mid X \in \cL, \; \rank(X) = n  \}} \subset \SSS^n$$
and call its projectivization $\PP \cL^{-1} \subset \PP \SSS^n$ the reciprocal variety of $\cL$.
The definition of $\cL^{-1}$ makes sense if $\cL$ contains at least one full-rank matrix.
We call such a linear space $\cL$ \emph{regular}.
We provide the following exact characterization of ML-maximal linear spaces in terms of the intersection of their reciprocal varieties and their polar spaces.
A formula for the ML-degree of $\cL$ in terms of Segre classes of this intersection is given in~\cite{AGKMS}.

\begin{thm}\label{thm: MLdefective equals non-empty intersection}
The ML-degree of a linear subspace $\cL \subset \SSS^n$ is at most the degree of its reciprocal variety.
Moreover, $\cL$ is ML-maximal (i.e., its ML-degree equals $\deg(\PP \cL^{-1})$)
if and only if $\cL^{-1} \cap \cL^\perp = \{ 0 \}$.
\end{thm}

In addition, we give alternative sufficient and necessary conditions for a linear space to be ML-maximal; see Remark~\ref{rem:ness+suff}.

Since generic linear spaces of symmetric matrices are ML-maximal, we want to study the fine structure of the complementary property.
For integers $k$ and $n$, we are interested in 
the set of all $k$-dimensional linear subspaces of $\SSS^n$ that are \emph{not} ML-maximal.
This subset of the Grassmannian $\Gr(k, \SSS^n)$ is neither Zariski closed nor open (see Remark~\ref{rem:notOpenNotClosed}).
Therefore, we study its Zariski closure $\NM_{k,n}$ in $\Gr(k, \SSS^n)$.

We show that this variety equals the \emph{bad locus} $\Bad_{k,n}$ studied in \cite[Section~3]{JS}.
The authors consider $k$-dimensional linear subspaces $\cL$ of the real symmetric $n \times n$ matrices $\SSS^n_{\mathbb{R}}$ 
and corresponding projections $\SSS^n_{\mathbb{R}} \to \mathrm{Hom}(\cL, \mathbb{R})$ dual to the inclusion $\cL \subset \SSS^n_{\mathbb{R}}$.
Such a linear space $\cL$
is called \emph{bad} if the image of the cone of positive semidefinite matrices
under the projection $\SSS^n_{\mathbb{R}} \to \mathrm{Hom}(\cL, \mathbb{R})$ is not closed.
Bad subspaces of $\SSS^n_{\mathbb{R}}$ are those for which strong duality in semidefinite programming fails, which has been thoroughly studied by Pataki~\cite{liu2015exact,pataki2013strong,pataki,pataki2019characterizing}.
The bad locus $\Bad_{k,n}$ is the Zariski closure in $\Gr(k, \SSS^n)$ of the set of $k$-dimensional bad subspaces of $\SSS^n_{\mathbb{R}}$.

\begin{thm}
\label{thm:nonMaximalEqualsBad}
The Zariski closure $\NM_{k,n}$ in $\Gr(k, \SSS^n)$ of the set of non-ML-maximal $k$-dimensional linear subspaces of $\SSS^n$ equals the 
bad locus $\Bad_{k,n}$.
\end{thm}

We describe the irreducible components of $\NM_{k,n}$ in terms of the 
determinantal varieties $D_s$ of matrices of rank at most $s$. 
The \emph{coisotropic variety} in $\Gr(k, \SSS^n)$ associated to $D_s$  is the Zariski closure of the set of all $k$-dimensional linear subspaces of $\SSS^n$ that intersect  $D_s$ at some smooth point non-transversely.
For all $s$ as in the following corollary, the coisotropic variety associated to $D_s$ has codimension one in $\Gr(k, \SSS^n)$~\cite[Corollary~6]{Ko}.

\begin{cor}
\label{cor:coisotropicUnion}
The non-ML-maximal locus $\NM_{k,n}$ is the union of the coisotropic hypersurfaces in $\Gr(k, \SSS^n)$ 
associated to the determinantal varieties $D_s$, where $s$ ranges over the integers such that $\binom{n-s+1}{2} < k \leq \binom{n+1}{2} - \binom{s+1}{2}$.

In particular, a generic linear space $\cL \in \Gr(k, \SSS^n)$ is ML-maximal.
\end{cor}

\begin{proof}
By \cite[Theorem~11]{JS}, the bad locus $\Bad_{k,n}$ is the union of the coisotropic hypersurfaces described above.
Hence, Theorem~\ref{thm:nonMaximalEqualsBad} implies the assertion.
\end{proof}

\begin{rem}\label{rem:general known}
Theorem~\ref{thm:nonMaximalEqualsBad} and Corollary~\ref{cor:coisotropicUnion} provide a geometric proof for the ML-maximality of generic linear spaces of symmetric matrices.
An alternative argument, using different techniques from commutative algebra, is given in~\cite[Theorem~2.3]{SU10}.
ML-maximality for generic linear spaces has also been conjectured in a more general setting in \cite[Conjecture 5.8]{MSUZ16}, and a positive answer has since been known to follow from a result of Teissier.
However, this has not been written down in the current literature, and therefore we include this argument in Section \ref{sec: Teissier}.
We also note that Theorem~\ref{thm: MLdefective equals non-empty intersection} is in fact a special case of~\cite[Theorem~5.5]{MSUZ16}, but we provide  a more detailed argument.
\end{rem}

We prove Theorem~\ref{thm: MLdefective equals non-empty intersection} in Section~\ref{sec:mldegree}
and Theorem~\ref{thm:nonMaximalEqualsBad} in Section~\ref{sec:zariskiClosure}.

\section{Maximum likelihood estimation}
\label{sec:mldegree}
In this section we prove Theorem \ref{thm: MLdefective equals non-empty intersection}. The \emph{maximum likelihood degree (ML-degree)} of a real linear space $\cL \subset \SSS^n$ 
is the number of complex critical points of the log-likelihood function
\begin{align*}
    \ell_S: \cL & \longrightarrow \mathbb{R}, \\
    X &\longmapsto \log \det(X) - \tr(SX)
\end{align*}
for a generic matrix $S \in \SSS^n$.
Our main tool to prove Theorem~\ref{thm: MLdefective equals non-empty intersection} is the projection away from~$\mathcal{L}^\perp:$
\begin{align*}
    \pi_{\mathcal{L}^\perp} : \mathbb{P} \mathbb{S}^n &\,\dashrightarrow \left\lbrace
    \mathcal{K} \in \mathrm{Gr}(\dim \mathcal{L}^\perp + 1,  \mathbb{S}^n) \mid  \mathcal{L}^\perp \subset \mathcal{K}
    \right\rbrace \cong \mathbb{P}^{\dim \mathbb{P}\mathcal{L}}, \\
    S &\longmapsto L_S := \mathrm{span}\lbrace \mathcal{L}^\perp, S \rbrace.
\end{align*}
In~\cite{AGKMS} it is shown that the ML-degree of  $\mathcal{L}$ is the degree of this projection restricted to the reciprocal variety $\mathbb{P} \mathcal{L}^{-1}$. 
In other words, the ML-degree of $\mathcal{L}$ is the cardinality of the generic fiber of the restricted projection $\pi_{\mathcal{L}^\perp} |_{\mathbb{P} \mathcal{L}^{-1}}$:
\begin{align}
\label{eq:genFiber}
\PP \left( L_S \cap \mathcal{L}^{-1}   \setminus \mathcal{L}^\perp \right) \text{ for generic } S \in \PP \mathbb{S}^n.  
\end{align}

\begin{proof}[Proof of Theorem \ref{thm: MLdefective equals non-empty intersection}.]
Let us first assume that we have $\PP \mathcal{L}^{-1} \cap \PP \mathcal{L}^\perp = \emptyset$.
Then it follows from the above that the ML-degree of $\mathcal{L}$ 
is the cardinality of the intersection $\PP L_S \cap \PP \mathcal{L}^{-1}$ for generic $ S \in \PP \mathbb{S}^n $.
Since the dimension of the projective space $\PP L_S$ is the codimension of $\PP \mathcal{L}^{-1}$, they intersect in either $\deg(\PP \mathcal{L}^{-1})$ many points (counted with multiplicity) or in infinitely many points.
The latter cannot happen for generic $S$, since domain and codomain of the map $\pi_{\mathcal{L}^\perp} |_{\mathbb{P} \mathcal{L}^{-1}}$ have the same dimension, 
so its generic fiber~\eqref{eq:genFiber} must be finite.
Thus, the intersection $\PP L_S \cap \PP \mathcal{L}^{-1}$ consists of $\deg(\PP \mathcal{L}^{-1})$ many points, counted with multiplicity. 
In~\cite{AGKMS} it is shown that the generic fiber~\eqref{eq:genFiber} is reduced,
so $\PP L_S \cap \PP \mathcal{L}^{-1}$ consists of $\deg(\PP \mathcal{L}^{-1})$ distinct points for generic $S$,
and we conclude that we have ML-degree$(\mathcal{L}) = \deg(\PP \mathcal{L}^{-1})$.

Conversely, we assume that the intersection $\PP \mathcal{L}^{-1} \cap \PP \mathcal{L}^\perp$ is non-empty. If it is finite, then, by the same reasoning as before, the intersection $\PP L_S \cap \PP \mathcal{L}^{-1}$ is again finite for generic $S$, and thus must consist of $\deg(\PP \mathcal{L}^{-1})$ many points (counted with multiplicity).
Since $\PP \mathcal{L}^{-1} \cap \PP \mathcal{L}^\perp \neq \emptyset$, we see that 
$\PP L_S \cap \PP \mathcal{L}^{-1}$ consists of strictly more points than~\eqref{eq:genFiber}.
All in all, we have for generic $S$ that
\begin{align*}
    \text{ML-degree}(\mathcal{L}) = |\PP \left(  L_S \cap \left(\mathcal{L}^{-1} \setminus \mathcal{L}^\perp\right) \right) |
    < |\PP L_S \cap \PP \mathcal{L}^{-1}| 
    \leq \deg(\PP \mathcal{L}^{-1}).
\end{align*}

Hence, we are left to consider the case when the intersection $\PP \mathcal{L}^{-1} \cap \PP \mathcal{L}^\perp$ is infinite.
Since the generic fiber~\eqref{eq:genFiber} is finite, the intersection $\PP L_S \cap \PP \mathcal{L}^{-1}$ consists of positive-dimensional components inside $\PP \mathcal{L}^\perp$ as well as $k$ points, among which those outside of $\PP \mathcal{L}^\perp$ contribute to the ML-degree of $\mathcal{L}$.
From the following standard fact from projective geometry it follows that we have $\text{ML-degree}(\mathcal{L}) \leq k < \deg(\PP \mathcal{L}^{-1})$.
\end{proof}

\begin{prop}
Let $X \subset \PP^m$ be an irreducible projective variety of degree~$d$, and let $L \subset \PP^m$ be a projective subspace of complementary dimension (i.e., $\dim L = m-\dim X$). 
If the intersection of $X$ and $L$ consists of positive dimensional components and $k$ points, then we have $k<d$.
\end{prop}

\begin{proof}
The following argument is due to Kristian Ranestad.

If $k=0$, there is nothing to show.
Hence, we assume from now on that $k>0$.
Let $Y_1$ be the union of all maximal-dimensional irreducible components of the intersection $X \cap L$.
We denote the remaining lower-dimensional components by $Z_1$ (i.e. $X \cap L = Y_1 \cup Z_1$), so in particular $Z_1$ contains the $k$ points.
The codimension $c_1$ of $Y_1$ in $X$ satisfies $0 < c_1 < \dim X$.
We consider a general projective space $L_1 \subset  \PP^m$ of codimension $c_1$ that contains $L$.
All components in the intersection $X \cap L_1$ have codimension $c_1$ in $X$. 
The latter can be seen by iteratively intersecting $X$ with general hyperplanes $H_1, \ldots, H_{c_1}$ containing $L$:
at each step, the irreducible components of $X \cap H_1 \cap \ldots \cap H_{i-1}$ are not contained in $H_i$ such that intersecting with $H_i$ reduces the dimension by $1$.

Some of the irreducible components in $X \cap L_1$ form $Y_1$.
We denote the remaining components by $X_1$ (i.e., $X \cap L_1 = Y_1 \cup X_1$).
Recall that $d$ is the degree of $X$; from the equality $d = \deg (X \cap L_1) = \deg Y_1 + \deg X_1$
it follows that $\deg X_1<d$.
Moreover, from
$Y_1 \cup Z_1 = X \cap L = (Y_1 \cup X_1) \cap L = Y_1 \cup (X_1 \cap L),$
it follows that we have $X_1 \cap L = (Y_1 \cap X_1) \cup Z_1$.
In particular, the intersection $X_1 \cap L$ contains the $k$ points.
We also note that $X_1$ and $L$ have complementary dimension inside $L_1$:
$\mathrm{codim}_{L_1}(X_1) = \dim L_1 - \dim X_1 = (m-c_1) - (\dim X - c_1)
= m - \dim X = \dim L$.
Hence, if the intersection $X_1 \cap L$ is finite, then we have shown that
$k \leq \deg X_1 < d$, so we are done.

Otherwise, if the intersection $X_1 \cap L$ is not finite, 
we consider the union $Y_2$ of the maximal-dimensional irreducible components of $X_1 \cap L$.
We observe that $\dim X_1 < \dim X$ and $\dim Y_2 < \dim Y_1$.
Now we repeat our construction above:
We let $Z_2$ be the remaining irreducible components of $X_1 \cap L = Y_2 \cup Z_2$, choose a general projective space $L_2 \subset \PP^m$ of codimension $c_2 := \mathrm{codim}_{X_1}(Y_2)$ that contains $L$, and denote by $X_2$ the irreducible components away from $Y_2$ in $X_1 \cap L_2 = Y_2 \cup X_2$.
If $X_2 \cap L$ is finite, the same arguments as above show that $k \leq \deg X_2 < d$.
Otherwise,  since $\dim X_2 < \dim X_1$ and $\dim (X_2 \cap L) < \dim Y_2$, 
we can repeat the above process several times until eventually $X_i \cap L$ will be finite for some $i \in \mathbb{N}$. 
At that point we can conclude the proof as $k \leq \deg X_i < d$.
\end{proof}

\section{Generic ML-maximality} \label{sec: Teissier}

In this section we show that a general linear space of symmetric matrices is ML-maximal. This result is not new, see Remark \ref{rem:general known}. What we show here is the equivalent statement (by Theorem \ref{thm: MLdefective equals non-empty intersection}) that for a generic linear space $\mathcal{L}$, we have $\mathcal{L}^{-1}\cap \mathcal{L}^{\perp}=\{0\}$. This was conjectured in more generality in \cite[Conjecture 5.8]{MSUZ16}, and shown in even more generality to follow from a statement by Teissier in \cite[Corollary 2.6]{MMW}. However, the authors do not write down how this follows exactly, which is why we include it here. It was explained to us by Mateusz Micha\l{}ek. 

For a positive integer $n$, we denote by $I_n$ the identity matrix of rank $n$. For an $n\times n$-matrix $X$ we denote by $\adj(X)$ its adjugate; we have $X\adj(X)=\det(X)\cdot I_n$, so if $X$ is invertible we have 
\begin{equation}\label{eq:adjugate}X^{-1}=(\det(X))^{-1}\cdot \adj(X).\end{equation} 

\begin{lemma}\label{lem:teissier}
Let $V$ be a complex vector space of dimension $n$ with dual space $V^*$, and $L\subset V$ a linear subspace. Let $L^{\perp}\subset V^*$ be the space of all linear forms that vanish on $L$. Moreover, let $f$ be a homogeneous polynomial on $V$, and ${\nabla}f$ its gradient map. If $L$ is generic, we have $\overline{\PP({\nabla}f)(L)}\cap\mathbb{P}L^{\perp}=\emptyset,$ where $\overline{\PP({\nabla}f)(L)}$ is the Zariski closure of $\PP({\nabla}f)(L)$ in~$\PP V^*$.
\end{lemma}

\begin{proof}The following is due to Mateusz Micha\l{}ek. 

Let $k$ be the dimension of $L$.
We may choose coordinates $\left(x_1,\ldots,x_n\right)$ on $V$ and an inner product to identify the dual space $V^*$ with $V$ such that
$L$ is defined by the equations $x_{k+1}=\cdots=x_n=0$, and $L^{\perp}$ is given by $x_1=\cdots=x_k=0$.  
Assume that $L$ is generic, and assume by contradiction that there is a sequence of elements $(X_j)_{j\geq1}$ in $\PP(\nabla f)(L)$ with limit contained in $\PP L^{\perp}$. 
Let $(Y_j)_{j\geq1}$ be a sequence in $\PP L$ such that we have 
$$(X_j)_{j\geq1}=\left(\left(\tfrac{\partial f}{\partial x_1}(Y_j) : \ldots : \tfrac{\partial f}{\partial x_n}(Y_j)\right)\right)_{j\geq1}.$$ Then $\lim_{j\rightarrow\infty} X_j\in\PP L^{\perp}$ implies $\lim_{j\rightarrow\infty}\tfrac{\partial f}{\partial x_i}(Y_j)=0$ for $i\in\{1,\ldots,k\}$, while $\lim_{j\rightarrow\infty}\tfrac{\partial f}{\partial x_l}(Y_j)\neq0$ for at least one $l\in\{k+1,\ldots,n\}$. Fix such an $l$. Since $L$ is generic, it follows from  \cite[II.2.1.3]{Teissier2} that $\tfrac{\partial f}{\partial x_l}$ is integral over the ideal $I=\left(\tfrac{\partial f}{\partial x_1},\ldots,\tfrac{\partial f}{\partial x_k}\right)$ in the ring $\mathbb{C}[x_1,\ldots,x_n]$.
By definition, the latter means that there is an integer $p$, and elements $a_i\in I^{p-i}$ for $i\in\{1,\ldots,p-1\}$, such that $$\left(\tfrac{\partial f}{\partial x_l}\right)^p+\sum_{i=0}^{p-1}a_i\left(\tfrac{\partial f}{\partial x_l}\right)^i=0.$$ 
Plugging in $Y_j$ and taking the limit, we find
$$
0=\lim_{j\rightarrow\infty}\left(\tfrac{\partial f}{\partial x_l}(Y_j)\right)^p+\lim_{j\rightarrow\infty}\sum_{i=0}^{p-1}a_i(Y_j)\left(\tfrac{\partial f}{\partial x_l}(Y_j)\right)^i=\lim_{j\rightarrow\infty}\left(\tfrac{\partial f}{\partial x_l}(Y_j)\right)^p,$$
where the last equality follows from the fact that $a_i$ is contained in $I^{p-i}$. But this contradicts our assumption that $\lim_{j\rightarrow\infty}\tfrac{\partial f}{\partial x_l}(Y_j)\neq0$.
We conclude that there is no sequence of elements in $\PP(\nabla f)(L)$ with limit contained in $\PP L^{\perp}$ if $L$ is generic. This finishes the proof.  
\end{proof}

\begin{cor}\label{cor:generic disjoin intersection}
For generic linear subspaces $\cL\subset\SSS^n$, we have $\PP\cL^{-1}\cap\PP\cL^{\perp}=\emptyset.$
\end{cor}
\begin{proof}
Let $f$ be the map $f\colon\mathbb{S}^n\longrightarrow\mathbb{C}$ given by $X\longmapsto \det(X).$ By the Jacobi formula and the fact that we work with symmetric matrices, for $X\in\mathbb{S}^n$ we have $({\nabla}f)(X)=\adj(X)\in\SSS^n$. From this and (\ref{eq:adjugate}) it follows that for a linear subspace in $\SSS^n$ we have $$\PP\mathcal{L}^{-1}\cap\mathbb{P}\mathcal{L}^{\perp}=\overline{\mathbb{P}({\nabla}f)(\mathcal{L})}\cap\mathbb{P}\mathcal{L}^{\perp},$$ 
where $\overline{\mathbb{P}({\nabla}f)(\mathcal{L})}$ is the Zariski closure of $\mathbb{P}({\nabla}f)(\mathcal{L})$ in $\mathbb{P}\mathbb{S}^n$. 
The statement now follows from Lemma \ref{lem:teissier}.
\end{proof}

\begin{rem}As we will show in Lemma \ref{lem:open set contained in A}, for the determinant map $f$ as in Corollary \ref{cor:generic disjoin intersection} the image $\PP({\nabla}f)(\mathcal{L})$ is disjoint from $\PP\mathcal{L}^{\perp}$ if and only if $\cL$ is not contained in the hyperplane tangent to the zero locus $Z(f)$ of $f$ at a smooth point belonging $\cL$ (this also follows from the more general statement for any \textsl{hyperbolic} polynomial in \cite[Proposition 5.9]{MSUZ16}). But this is true for generic $\cL$ by Bertini's Theorem \cite[Theorem 17.16]{Harris}. 
The added value of Lemma~\ref{lem:teissier} is thus to show that the \textsl{closure} of $\mathbb{P}({\nabla}f)(\mathcal{L})$ is disjoint from $\PP\cL^{\perp}$ for generic~$\cL$.
As we saw in the proof of the lemma, where Bertini shows that the radical of the ideal of the singular locus  Sing($Z(f)\cap\cL$) of $Z(f)\cap\cL$ equals the radical of the ideal of Sing$(Z(f))\cap \cL$, Teissier shows that these two ideals are in fact \textsl{integral} over each other. There are several versions of Teissier's `Th\'{e}or\`{e}me Bertini id\'{e}aliste'; see  \cite[Proposition 2.7]{Teissier1}, \cite[p.42]{Teissier3},  \cite[Theorem 2.6]{Gaffney}. \end{rem}

\section{Sufficient and necessary conditions for ML-maximality}
Let $k$ and $n$ be two integers. By $\NM^\circ_{k,n}$ we denote the set in Gr$(k,\SSS^n)$ of $k$-dimensional linear subspaces in $\SSS^n$ that are not ML-maximal. Note that NM$_{k,n}$ is its Zariski closure by definition, and by Theorem \ref{thm: MLdefective equals non-empty intersection}, we have 
$$\NM^\circ_{k,n}= \{\mathcal{L}\in \mbox{Gr}(k,\mathbb{S}^n) \mid \PP\mathcal{L}^{-1}\cap\mathbb{P}\mathcal{L}^{\perp}\neq\emptyset\}.$$ 
In this section we show that $\NM^\circ_{k,n}$ is neither Zariski open nor closed (Remark \ref{rem:notOpenNotClosed}), and we provide 
sufficient and necessary conditions for a linear space to be ML-maximal (Remark \ref{rem:ness+suff}).
More specifically, we describe a subset of $\NM^\circ_{k,n}$ in terms of tangency to the determinantal hypersurface (Lemma \ref{lem:open set contained in A}), and we show that $\NM^\circ_{k,n}$ is contained in the closed set 
\begin{align}
\label{eq:Ckn}
    C_{k,n}:=\{\mathcal{L}\in \mbox{Gr}(k,\mathbb{S}^n)\mid\exists \;(X,Y)\in \mathbb{P}\mathcal{L}\times\mathbb{P}\mathcal{L}^{\perp}\colon X Y=0\}
\end{align}
(Corollary \ref{cor:AcontainedinC}). The latter is
one of the main ingredients in the proof of Theorem~\ref{thm:nonMaximalEqualsBad} (see Section \ref{sec:zariskiClosure}).
The set $C_{k,n}$ is the union of the coisotropic hypersurfaces in $\Gr(k, \SSS^n)$ associated to the determinantal variety $D_s$, where $s$ ranges over the integers such that $\binom{n-s+1}{2} \leq k \leq \binom{n+1}{2} - \binom{s+1}{2}$~\cite[Theorems~11 and~17]{JS}.

\begin{rem}\label{rem:notOpenNotClosed}
The set $\NM^\circ_{k,n}$ is in general neither Zariski open nor closed. To illustrate this, 
we consider the stratification of the Grassmannian $\Gr(2, \SSS^n)$ in terms of Segre symbols described in~\cite[Section~5]{FMS}.
In \cite[Example 1.3]{FMS}, we see that the ML-maximal elements of $\Gr(2,\mathbb{S}^3)$ lie in the strata of $\Gr(2,\mathbb{S}^3)$ defined by Segre symbols with only 1's. 
In other words, the complement of $\NM^\circ_{2,3}$ is the union of the two strata $\Gr_{[1,1,1]}$ and $\Gr_{[(1,1),1]}$. 
However, in Figure~1 of the same paper, we find the following inclusions of Zariski closures of strata of codimensions $2$, $1$ and $0$ in $\Gr(2, \SSS^3)$:
$$\overline{\Gr_{[(1,1),1]}} \subset \overline{\Gr_{[2,1]}} \subset \overline{\Gr_{[1,1,1]}} =  \Gr(2,\mathbb{S}^3).$$ 
We conclude that the complement of $\NM^\circ_{2,3}$ in $\Gr(2,\SSS^3)$ is neither Zariski open nor closed, hence neither is $\NM^\circ_{2,3}$. 
By \cite[Example 3.1]{FMS},
the ML-maximal elements of $\Gr(2, \SSS^4)$ lie again in the strata with Segre symbols containing only 1's.
The same argument shows that $\NM^\circ_{2,4}$ is neither Zariski open nor closed.
\end{rem}



\begin{lemma}
\label{lem:notFullRankIntersection}
Let $\cL \subset \SSS^n$ be a regular linear space.
The intersection $\cL^{-1} \cap \cL^\perp$ does not contain matrices of full rank.
\end{lemma}

\begin{proof}
If the intersection would contain a full-rank matrix $Y$, then $Y^{-1} \in \cL$ and $Y \in \cL^\perp$, which yields $0 = \tr(Y^{-1}Y) = \tr(I_n) = n > 0$; a contradiction.
\end{proof}

For a linear space $\cL$ in $\SSS^n$, we denote by $\PP\mathring{\mathcal{L}}^{-1}$ the open subset of $\PP\cL^{-1}$ given by the adjugates of all matrices in $\PP\cL$ of rank at least $n-1$. 
We now describe the following  subset of $\NM^\circ_{k,n}$: \begin{equation}\label{eq:open set}\{\mathcal{L}\in \mbox{Gr}(k,\mathbb{S}^n) \mid \PP\mathring{\mathcal{L}}^{-1}\cap\mathbb{P}\mathcal{L}^{\perp}\neq\emptyset\}.\end{equation}
The next lemma also follows from the more general \cite[Proposition 5.9]{MSUZ16}.
\begin{lemma}\label{lem:open set contained in A}
The  subset of $\NM^\circ_{k,n}$ given by (\ref{eq:open set}) is equal to 
$$\{\mathcal{L}\in \Gr(k,\mathbb{S}^n) \mid\exists\;X\in\mathrm{Reg}(D_{n-1})\colon X \in  \mathcal{L}\subset T_X(D_{n-1})\}.$$
\end{lemma}

\begin{proof} If $\mathcal{L}$ is contained in $(\ref{eq:open set})$, then there is a matrix $X$ in $\PP\mathcal{L}$ of rank at least $n-1$, whose adjugate is contained in $\mathbb{P}\mathcal{L}^{\perp}$. 
By Lemma~\ref{lem:notFullRankIntersection}, $X$ has rank $n-1$, hence it is a regular point of $D_{n-1}$. Since $\adj(X)$ is in the annihilator of the tangent hyperplane $T_X(D_{n-1})$, and $\adj(X)$ is contained in $\mathbb{P}\mathcal{L}^{\perp}$, it follows that $\mathcal{L}$ is a subset of $T_X(D_{n-1})$. Conversely, let $X$ be a regular point of $D_{n-1}$ such that $\mathcal{L}$ is tangent to $D_{n-1}$ at $X$.
Then $X$ has rank $n-1$, so $\adj(X)$ is contained in $\PP\mathring{\cL}^{-1}$, and it is in the annihilator of $T_X(D_{n-1})$, so it is contained in $\mathbb{P}\mathcal{L}^{\perp}$. 
\end{proof}

\begin{rem}
\label{rem:smallCoisotropicVariety}
Lemma~\ref{lem:open set contained in A} says that~\eqref{eq:open set} is exactly the set of $k$-dimensional tangent spaces at smooth points of the determinantal hypersurface $D_{n-1}$. 
Thus, its Zariski closure is the irreducible coisotropic variety Ch$_{k-1}(D_{n-1})$. \cite{Ko}
\end{rem}

\begin{exa}
The following linear space is an element in $\NM^\circ_{3,3}$ that is not contained in its subset~\eqref{eq:open set}. 
Let $\cL$ be spanned by $\left[\begin{smallmatrix}
1&0&0\\
0&0&0\\
0&0&0
\end{smallmatrix}\right], \left[\begin{smallmatrix}
0&0&1\\
0&1&0\\
1&0&0
\end{smallmatrix}\right],\left[\begin{smallmatrix}
0&0&0\\
0&0&1\\
0&1&0
\end{smallmatrix}\right]$.
The intersection $\PP\mathcal{L}^{-1} \cap \mathbb{P}\mathcal{L}^{\perp}$ consists of the single element $M=\left[\begin{smallmatrix}
0&0&0\\
0&0&0\\
0&0&1
\end{smallmatrix}\right]$, so $\mathcal{L}$ is contained in $\NM^\circ_{3,3}$. However, $\PP\mathring{\mathcal{L}}^{-1}\cap\PP\mathcal{L}^{\perp}$ is empty, since $M$ is not the adjugate of any matrix in $\mathcal{L}$. 

The linear space $\cL$ is a net of conics of Type~C according to Wall's classification~\cite{wall}.
This means that it is a generic point in the intersection of the Chow hypersurface $\mathrm{Ch}_0(D_1)$ of the rank-one locus $D_1$ with the Zariski closure $\mathrm{Ch}_2(D_2)$ of \eqref{eq:open set} described in Remark~\ref{rem:smallCoisotropicVariety}.
By duality, $\cL \in \mathrm{Ch}_2(D_2)$ if and only if $\cL^\perp \in \mathrm{Ch}_0(D_1)$.
Hence $\cL$ is generic among all linear spaces with the property that both $\cL$ and $\cL^\perp$ contain a rank-one matrix.
We also see from \cite[Table 1]{MLEnetsOfConics} that ML-degree$(\cL) = 2 < 3 = \deg(\PP \cL^{-1})$, so $\cL$ is not ML-maximal.
\end{exa}

We end this section by showing that $\NM^\circ_{k,n}$ is contained in the set $C_{k,n}$  defined in~\eqref{eq:Ckn}.
For any $\mathcal{L}\in G(k,\mathbb{S}^n)$, consider the Zariski closed set
$$C_{\mathcal{L}}=\{(X, Y)\in \mathbb{P}\mathcal{L} \times \mathbb{P}\mathbb{S}^n \mid X Y = t\cdot I_n\mbox{ for some }t\in\mathbb{C}\}$$
with the projection $\pi_\mathcal{L}: C_\mathcal{L} \to \mathbb{S}^n$ to the second coordinate. 

\begin{lemma}\label{lem:AcontainedinC}
Let $\cL \subset \SSS^n$ be a regular linear space.
For every $Y \in \cL^{-1}$ with $\rank(Y) < n$, there is an $X \in \cL,\; X \neq 0$, such that $XY = 0$.
\end{lemma}
\begin{proof}
For any $Z \in \mathbb{P}\mathcal{L} $ of full rank, we have $(Z, Z^{-1}) \in C_\mathcal{L}$, so $Z^{-1} \in \pi_{\mathcal{L}} C_\mathcal{L}$. Since $C_{\mathcal{L}}$ is a projective variety, $\pi_\mathcal{L}C_\mathcal{L}$ is closed, so $\mathcal{L}^{-1}$ is contained in $\pi_{\mathcal{L}} C_\mathcal{L}$.
Therefore, if $Y$ is in $\mathcal{L}^{-1}$, then $Y$ is in the image of $\pi_{\cL}$. So there is an $X\in\mathcal{L}$ with $XY=t\cdot I_n$ for some constant $t$. If $Y$ is not of full rank, this implies $XY=0$. 
\end{proof}

\begin{cor}
\label{cor:AcontainedinC}
The set $\NM^\circ_{k,n}$ is contained in $C_{k,n}$. 
\end{cor}

\begin{proof}
Let $\cL \in \NM^\circ_{k,n}$ be a regular subspace of $\SSS^n$.
By definition, there is a non-zero matrix $Y$ in the intersection $ \cL^{-1} \cap \cL^\perp$ which has rank $<n$ by Lemma \ref{lem:notFullRankIntersection}. Lemma \ref{lem:AcontainedinC} now guarantees the existence of a non-zero matrix $X \in \cL$ such that $XY=0$, which shows that $\cL$ is contained in $C_{k,n}.$ 
\end{proof}

\begin{rem}
\label{rem:involution}
The Zariski closure $\NM_{k,n}$ of $\NM^\circ_{k,n}$ is in general not equal to $C_{k,n}$, as can be seen by the following argument:
under the involution $\mathcal{L}\mapsto\mathcal{L}^{\perp}$, the set $C_{k,n}$ gets mapped to the set $C_{\binom{n+1}{2}-k,n}$, but $\NM_{k,n}$ is in general not mapped to $\NM_{\binom{n+1}{2}-k,n}$ as Example~\ref{exa:NMfailPolarity} illustrates.

In fact, by Theorem~\ref{thm:nonMaximalEqualsBad} and \cite[Theorem 17]{JS}, we have the following:
if there is an integer $s$ such that $k = \binom{n-s+1}{2}$, then $C_{k,n}$ is the union of $\NM_{k,n}$ and the Chow hypersurface $\mathrm{Ch}_0(D_s)$ of $D_s$; 
otherwise, $C_{k,n} = \NM_{k,n}$. 
\end{rem}

\begin{exa}\label{exa:NMfailPolarity}
The integers $k=3, n=3, s=1$ satisfy $k = \binom{n-s+1}{2}$.
The linear space $\cL$  in $\SSS^3$ spanned by diagonal matrices is contained in $C_{3,3}$ but not in $\NM_{3,3}$. 
According to Wall's classification~\cite{wall}, $\cL$ is a net of conics of Type $E$.
Its projectivization is a trisecant plane of the rank-one locus $D_1$, so in particular $\cL$ is in the Chow hypersurface $\mathrm{Ch}_0(D_1)$.
However, $\cL$ is not contained in the coisotropic hypersurface $\mathrm{Ch}_2(D_2)$, which is by Corollary~\ref{cor:coisotropicUnion} equal to $\NM_{3,3}$.

On the other hand, its polar net $\cL^\perp$ is of type $E^*$, so by \cite[Table 1]{MLEnetsOfConics} we have ML-degree$(\cL^\perp) = 1 < 4 = \deg(\PP (\cL^\perp)^{-1})$. Hence the polar net $\cL^\perp$ is not ML-maximal, i.e. $\cL^\perp \in  \NM^\circ_{3,3}$.
Corollary~\ref{cor:AcontainedinC} and Remark~\ref{rem:involution} imply that both $\cL^\perp$ and $\cL$ are contained in $C_{3,3}$.
\end{exa}

\begin{rem}\label{rem:ness+suff}
Lemma \ref{lem:open set contained in A} gives a sufficient condition, and Corollary \ref{cor:AcontainedinC} gives a necessary condition for a linear subspace not to be ML-maximal.
On the one hand, every linear space that is tangent to the manifold of corank-one matrices is not ML-maximal.
On the other hand, for every linear space $\cL$ that is not ML-maximal, there are non-zero matrices $X \in \cL$ and $Y \in \cL^\perp$ such that $XY=0$.
\end{rem}

\section{Proof of Theorem \ref{thm:nonMaximalEqualsBad}}
\label{sec:zariskiClosure}
Recall that $\Bad_{k,n}$ is the Zariski closure in $\Gr(k,\SSS^n)$ of the set of $k$-dimensional bad subspaces of $\SSS_{\mathbb{R}}^n$ as defined in the introduction.
In this section, we prove that $\NM_{k,n}$ equals $\Bad_{k,n}$. We start with one inclusion.

\begin{prop}
\label{prop:badContainedInMLdefective}
$\mathrm{Bad}_{k,n}$ is contained in
$\NM_{k,n}$.
\end{prop}

\noindent
To show the proposition, we outsource all the hard work to the following lemma.

\begin{lemma}
\label{lem:pointsInLinvernse}
Let $s$ be an integer with $0 < s < n$ and $k > \binom{n-s+1}{2}$.
For fixed   $X,Y \in \mathbb{S}^n$ with $\rank(X)=s$, $\rank(Y) = n-s$ and $XY= 0$, we consider the variety
\begin{align*}
    \mathcal{G}_{X,Y} := \left\lbrace \mathcal{L} \in \mathrm{Gr}(k,\mathbb{S}^n) \mid X \in \mathcal{L}, Y \in \mathcal{L}^\perp \right\rbrace.
\end{align*}
A general $\mathcal{L}$ in $\mathcal{G}_{X,Y}$
satisfies
\begin{align*}
    \left\lbrace Z \in \mathbb{S}^n \mid  XZ=0 \right\rbrace \subseteq \mathcal{L}^{-1};
\end{align*} 
in particular, we have that $Y$ is contained in $\mathcal{L}^{-1}$.
\end{lemma}

We first prove Proposition~\ref{prop:badContainedInMLdefective} to see how we can apply the lemma.
Afterwards, we give the proof of Lemma~\ref{lem:pointsInLinvernse}.

\begin{proof}[Proof of Proposition~\ref{prop:badContainedInMLdefective}]
The bad locus $\mathrm{Bad}_{k,n}$ is the union of the irreducible coisotropic hypersurfaces in $\mathrm{Gr}(k, \mathbb{S}^n)$ associated to the bounded-rank loci $D_s$ where $s$ is in the range $\binom{n-s+1}{2} < k \leq \binom{n+1}{2} - \binom{s+1}{2}$ \cite[Theorem 11]{JS}.
A general point $\mathcal{L}$ in one of these hypersurfaces satisfies the following property:
\begin{align*}
\exists\; X \in \mathcal{L} \cap \mathrm{Reg}(D_s): \mathcal{L} + T_X D_s \neq \mathbb{S}^n.
\end{align*}
This implies that $\mathcal{L}$ and the tangent space $T_X D_s$ have a common non-zero element $Y$ in their annihilators.
In other words, there is a non-zero matrix $Y \in \mathcal{L}^\perp$ satisfying $XY=0$.
The rank of that matrix $Y$ is at most $n-s$.
Since $\mathcal{L}$ is general, we may assume that $\rank(Y) = n-s$. 
In fact, we may choose $\mathcal{L}$ by first fixing any $X \in \mathcal{L}$  of rank~$s$, then fixing any $Y \in \mathcal{L}^\perp$ of rank $n-s$ with $XY=0$, and finally choosing the remaining basis vectors of $\mathcal{L}$ arbitrarily.
In other words, $\mathcal{L}$ is a general point of $\mathcal{G}_{X,Y}$,
so by Lemma~\ref{lem:pointsInLinvernse} 
we see that $Y$ is contained in $\mathcal{L}^{-1}$.
It follows that $\mathcal{L}$ is contained in $\NM_{k,n}$.
\end{proof}

In the proof of Lemma \ref{lem:pointsInLinvernse}, we compute the total transform of a point in the blow-up of a linear space along the indeterminacy locus of the adjugate map. 
Since this is a technical construction, we first do this in a concrete example.
It was shown in~\cite{graph}
that the blow-up of $\mathbb{P}\mathbb{S}^n$ along $D_{n-2}$, i.e. the Zariski closure of the graph of matrix inversion on $\mathbb{P}\mathbb{S}^n$, is
\begin{align*}
    \Gamma := \left\lbrace (X,Y) \in \mathbb{P}\mathbb{S}^n \times \mathbb{P}\mathbb{S}^n \mid XY = t \cdot I_n \text{ for some } t \in \mathbb{C} \right\rbrace.
\end{align*}
For a regular linear space $\cL$ in $\SSS^n$, we use the Zariski closure
$$\Gamma_\mathcal{L} := \overline{\left\{(X,X^{-1}) \in \PP\SSS^n \times \PP\SSS^n \mid X \in \PP \cL \right\}}$$ of the graph of matrix inversion restricted to $\mathbb{P}\mathcal{L}$ to understand the reciprocal variety $\PP\mathcal{L}^{-1}$, as $\PP\cL^{-1}$ is the image of the projection of $\Gamma_{\cL}$ to the second factor.
In particular, for a point $X \in \PP \cL$, we are interested in its \emph{total transform}
$$\Gamma_\cL(X) := \{ Z \in \PP \SSS^n \mid (X,Z) \in \Gamma_\cL \} \subseteq \{ Z \in  \PP \cL^{-1} \mid XZ= t \cdot I_n \text{ for some } t \in \mathbb{C}  \}.$$

\begin{exa}\label{ex:accompanyLemma}$(n=3,\;k=5)$ Let $X,Y\in\mathbb{S}^3$ be the matrices given by \[X=\begin{bmatrix}
1 & 0 & 0 \\
0 & 0 & 0 \\
0 & 0 & 0
\end{bmatrix},\;\;\;Y=\begin{bmatrix}
0 & 0 & 0 \\
0 & 1 & 0 \\
0 & 0 & 1
\end{bmatrix}.\] Let $\mathcal{L}$ be the polar linear space $Y^{\perp}$ of $Y$ in $\mathbb{S}^3$, and note that $X$ is contained in $\mathcal{L}$. 
Setting $s=1$, the matrices $X,Y$ satisfy the conditions in Lemma \ref{lem:pointsInLinvernse}, and $\mathcal{L}$ is contained in $\mathcal{G}_{X,Y}$. 
We compute that (the affine cone over) the total transform $\Gamma_\cL(X)$ is $\{Z\in\SSS^3\mid XZ=0\}$, so the latter is contained in~$\cL^{-1}$.

A basis for $\mathcal{L}$ is given by $\{ X,B_{01},B_{02},B_1,B_2\}$, where $\left(B_{01},B_{02},B_1,B_2\right)$ is $$
\left(\left[\begin{matrix}
0&0&1\\
0&0&0\\
1&0&0
\end{matrix}\right],\left[\begin{matrix}
0&1&0\\
1&0&0\\
0&0&0
\end{matrix}\right],\left[\begin{matrix}
0&0&0\\
0&1&0\\
0&0&-1
\end{matrix}\right],\left[\begin{matrix}
0&0&0\\
0&0&1\\
0&1&0
\end{matrix}\right]\right).$$
For a matrix $M\in\mathbb{S}^n$ we denote its lower-right $2\times2$ block by $\overline{M}$. Note that $\overline{B}_{01}$ and $\overline{B}_{02}$ are both 0, and $\overline{B_1}$, $\overline{B_2}$ together span $I_2^{\perp}$ in $\mathbb{S}^2$.
To determine the total transform $\Gamma_\cL(X)$, we perturb $X$, and then compute its adjugate. Let $\varepsilon$ be an indeterminate. The first perturbation we compute is $$X + \varepsilon ( b_{01} B_{01} + b_{02} B_{02} + b_1B_1+b_2B_2) =\left[\begin{matrix}
1&\varepsilon b_{02}&\varepsilon b_{01}\\
\varepsilon b_{02}&\varepsilon b_{1}&\varepsilon b_{2}\\
\varepsilon b_{01}&\varepsilon b_{2}&-\varepsilon b_{1}
\end{matrix}\right],$$
where $(b_{01},b_{02},b_1,b_2)$ is a vector in $\mathbb{C}^4\setminus\{0\}$. The adjugate of this matrix is $$\left[\begin{matrix}
-\varepsilon^2(b_1^2+b_2^2)&\varepsilon^2(b_{02}b_1+b_{01}b_2)&\varepsilon^2(b_{02}b_2-b_{01}b_1)\\
\varepsilon^2(b_{02}b_1+b_{01}b_2)&-\varepsilon( b_{1}+\varepsilon b_{01}^2)&-\varepsilon( b_{2}-\varepsilon b_{01}b_{02})\\
\varepsilon^2(b_{02}b_2-b_{01}b_1)&-\varepsilon( b_{2}-\varepsilon b_{01}b_{02})&\varepsilon( b_{1}-\varepsilon b_{02}^2)
\end{matrix}\right].$$ Note that the lowest degree terms are all in the $2\times2$ lower-right block. Dividing by $\varepsilon$ and setting $\varepsilon=0$, we obtain the matrix
$$\left[\begin{matrix}0&0&0\\
0&-b_{1}&-b_{2}\\
0&-b_{2}&b_{1}\end{matrix}\right]=
    \left[ \begin{array}{c|c}
        0 & 0 \\ \hline
        0 & \adj(b_1 \bar B_1 + b_2 \bar B_2)
    \end{array} \right].$$
Since $\overline{B}_1, \overline{B}_2$ span $I_2^{\perp}$, this implies $\left\lbrace Z \mid XZ = 0,\; \bar Z \in (I_{2}^\perp)^{-1} \right\rbrace \subseteq \Gamma_{\mathcal{L}}(X).$

The second perturbation of $X$ that we compute is the matrix $$X + \varepsilon ( c_{01} B_{01} + c_{02}\varepsilon B_{02} + c_1\varepsilon B_1+c_2\varepsilon B_2) =\left[\begin{matrix}
1&\varepsilon^2 c_{02}&\varepsilon c_{01}\\
\varepsilon^2 c_{02}&\varepsilon^2c_{1}&\varepsilon^2 c_{2}\\
\varepsilon c_{01}&\varepsilon^2 c_{2}&-\varepsilon^2 c_{1}
\end{matrix}\right],$$where $c=(c_{01},c_{02},c_1,c_2)$ is a vector in $\mathbb{C}^4\setminus\{0\}$. We find the adjugate
$$\left[\begin{matrix}
-\varepsilon^4(c_1^2+c_2^2)&\varepsilon^3(\varepsilon c_{02}c_1+ c_{01}c_2)&\varepsilon^3(\varepsilon c_{02}c_2-c_{01}c_1)\\
\varepsilon^3(\varepsilon c_{02}c_1+c_{01}c_2)&-\varepsilon^2( c_{1}+ c_{01}^2)&-\varepsilon^2( c_{2}-\varepsilon c_{01}c_{02})\\
\varepsilon^3(\varepsilon c_{02}c_2- c_{01}c_1)&-\varepsilon^2( c_{2}-\varepsilon c_{01}c_{02})&\varepsilon^2( c_{1}-\varepsilon^2 c_{02}^2)
\end{matrix}\right].$$Again, the lowest degree terms are in the $2\times2$ lower-right block. We now divide by $\varepsilon^2$ and set $\varepsilon=0$, and obtain $$Z_c=\left[\begin{matrix}
0&0&0\\
0&-(c_{1}+c_{01}^2)&- c_{2}\\
0&-c_{2}& c_{1}
\end{matrix}\right].$$Let $\mathcal{Z}$ be the closure in $\SSS^2$ of the set $\{\overline{Z}_c \mid c\in \mathbb{C}^4\setminus\{0\}\}$. 
All elements $\overline{Z}_c\in\mathcal{Z}$ with $c=(0,c_{02},c_1,c_2)$ parametrize the hypersurface $(I_2^{\perp})^{-1}$ in $\SSS^2$ as before, so  $(I_2^{\perp})^{-1}$ is contained in $\mathcal{Z}$. 
Since $\mathcal{Z}$ is irreducible, it follows that $\mathcal{Z}$ is either equal to $(I_2^{\perp})^{-1}$ or to $\mathbb{S}^2$. However, the element $\overline{Z}_{(1,0,1,0)}$ is contained in $\mathcal{Z}$, and$$\left(\overline{Z}_{(1,0,1,0)}\right)^{-1}=\left[\begin{matrix}
-\frac{1}{2}&0\\
0&1
\end{matrix}\right]\not\in I_2^{\perp},$$ 
so $\mathcal{Z}$ is not contained in $(I_2^{\perp})^{-1}$, hence $\mathcal{Z}=\mathbb{S}^2$. 
We have shown that every matrix $Z\in\mathbb{S}^{3}$ with $XZ=0$ is contained in $\mathcal{Z}$. 
Therefore,  $\Gamma_{\mathcal{L}}(X) = \{ Z \in \PP \SSS^3 \mid XZ=0 \}$.
\end{exa}

We now give the proof of Lemma \ref{lem:pointsInLinvernse}. We generalize the construction that was done in the previous example, and we recommend reading this example alongside the proof for illustration. 

\begin{proof}[Proof of Lemma~\ref{lem:pointsInLinvernse}]
After a change of coordinates, we may assume that 
\begin{align}
\label{eq:specialChoice}
    X = \left[ \begin{array}{c|c}
        I_s & 0  \\ \hline
        0  & 0
    \end{array} \right] \text{ and }
    Y = \left[ \begin{array}{c|c}
        0 & 0  \\ \hline
        0  & I_{n-s}
    \end{array} \right].
\end{align}
In the following, we will denote by $\bar{M}$ the lower-right $(n-s)\times(n-s)$ block of a symmetric matrix $M \in \mathbb{S}^n$.
For instance, $\bar X = 0$ and $\bar Y = I_{n-s}$.

A dimension count reveals that every $\mathcal{L} \in \mathcal{G}_{X,Y}$ must contain a matrix $B_0$ with $\bar B_0 = 0$ that is linear independent from $X$.
Indeed, the vector space of symmetric matrices M satisfying $\bar M=0$
is contained in the hyperplane $Y^\perp$ and
has codimension $\binom{n-s+1}{2}-1$ in $Y^\perp$.
As $\mathcal{L}$ is also contained in $Y^\perp$, its projection away from $X$ yields a $(k-1)$-dimensional vector space inside $Y^\perp$.
Since we have $k-1 \geq \binom{n-s+1}{2}$, that vector space  contains a non-zero matrix $B_0$ with $\bar B_0=0$.

The same dimension count also shows that the image of the projection \mbox{$M \mapsto \bar M$} onto $\mathbb{S}^{n-s}$ restricted to a general $\mathcal{L} \in \mathcal{G}_{X,Y}$ is the whole hyperplane $I_{n-s}^\perp$.
In terms of a basis $\lbrace X, B_0, B_1, \ldots, B_{k-2} \rbrace$ (with $\bar B_0=0$) for a general $\mathcal{L} \in \mathcal{G}_{X,Y}$ this means that $\bar{B}_1, \ldots, \bar B_{k-2}$ span the hyperplane $I_{n-s}^\perp$.

Equipped with this knowledge, we will now prove the assertion.
Note that, for $Z\in\mathbb{S}^n$ with $(X,Z)\in\Gamma_{\mathcal{L}}$, we have $XZ=0$. In what follows we show that the converse holds as well; more specifically, we will show that 
in fact \emph{all} pairs $(X,Z)$ with $Z\in\mathbb{S}^n$, $XZ=0$, are contained in $\Gamma_\mathcal{L}$.

Apply matrix inversion to the blow-up of $\mathcal{L}$ at $X$.
In terms of a basis $\lbrace X, B_0, B_1, \ldots, B_{k-2} \rbrace$  for $\mathcal{L}$,
this means to compute the matrix $Z_{b}$ that appears as the first non-zero coefficient of the following power series in $\varepsilon$:
\begin{align}
\label{eq:blowUp}
    \adj\left(X + \varepsilon ( b_0 B_0 + b_1 B_1 + \ldots + b_{k-2} B_{k-2}) \right),  
    \end{align}
where $b = (b_0, b_1, \ldots, b_{k-2})$ is a non-zero vector of arbitrary power series  $b_i$ in $\mathbb{C}[[\varepsilon]]$.
All matrices $Z_b$ obtained in this way satisfy $(X,Z_b) \in \Gamma_\mathcal{L}$. 
In what follows we show that the closure of the set of matrices $Z_b$ where $b$ is a non-zero vector of either constants, or with $b_0$ constant and the other $b_i$ linear monomials, already contains $\{ (X,Z) \mid Z\in\mathbb{S}^n, \; XZ = 0\}$, thus proving the lemma.

Let us first compute $Z_b$ for the case where $b$ is a non-zero vector in $\mathbb{C}^4$.
The lowest degree terms in the matrix~\eqref{eq:blowUp} are of degree $n-s-1$
and appear exactly in its lower right block.
Their coefficients are the minors of size $n-s-1$ of the matrix $b_1 \bar B_1 + \ldots + b_{k-2} \bar B_{k-2} $ (since $\bar B_0=0$). 
More precisely, we see that 
\begin{align*}
    Z_b = \left[ \begin{array}{c|c}
        0 & 0 \\ \hline
        0 & \adj(b_1 \bar B_1 + \ldots + b_{k-2} \bar B_{k-2})
    \end{array} \right].
\end{align*}
Due to the generality of $\mathcal{L}$, the matrices  $\bar{B}_1, \ldots, \bar B_{k-2}$ span the hyperplane $I_{n-s}^\perp$,
so the closure of the set $\{\bar Z_b \mid b\in\mathbb{C}^4\setminus\{0\}\}$ in $\mathbb{S}^{n-s}$ equals the reciprocal hypersurface  $(I_{n-s}^\perp)^{-1}$.
Hence, we have proven so far that
\begin{align*}
    \left\lbrace (X,Z) \mid XZ = 0,\; \bar Z \in (I_{n-s}^\perp)^{-1} \right\rbrace \subseteq \Gamma_{\mathcal{L}}.
\end{align*}

Finally, we compute $Z_b$ for the case when $b$ is a non-zero vector where $b_0$ a constant and the power series $b_1,\ldots, b_{k-2}$ have only linear terms (i.e., $b_i = c_i \varepsilon$ for a constant $c_i$).
The lowest degree terms in the matrix~\eqref{eq:blowUp} are of degree $2(n-s-1)$
and appear again only in its lower right block.
Now the coefficients of these terms do not only depend on $b_1, \ldots, b_{k-2}$, but also on~$b_0$. 
The closure of the set of the resulting $\bar Z_b$ forms an irreducible subvariety  $\mathcal{Z}$ of $\mathbb{S}^{n-s}$. 
Setting $b_0=0$, we see that $\mathcal{Z}$ contains the reciprocal hypersurface  $(I_{n-s}^\perp)^{-1}$.
Hence $\mathcal{Z}$ is either equal to $(I_{n-s}^\perp)^{-1}$ or it is the whole ambient space $\mathbb{S}^{n-s}$.
The condition ``$\mathcal{Z} \subseteq (I_{n-s}^\perp)^{-1}$''
is Zariski closed in the entries of the matrices $B_0, B_1, \ldots, B_{k-2}$.
Thus, if there is one instance with $\mathcal{Z} \not\subseteq (I_{n-s}^\perp)^{-1}$, 
then we know $\mathcal{Z} = \mathbb{S}^{n-s}$ for general choices of $B_0, \ldots, B_{k-2}$ (with $\bar B_0=0$). For general $\mathcal{L} \in \mathcal{G}_{X,Y}$ we can then conclude that $\lbrace (X,Z) \mid Z\in\mathbb{S}^n,\;XZ=0 \rbrace \subseteq \Gamma_\mathcal{L}$, which proves the assertion. 

We exhibit such an instance.
Since $\bar B_1,\ldots, \bar B_{k-2}$ span the hyperplane $I_{n-s}^\perp$, 
we may assume that the first $\binom{n-s+1}{2}-1$ of these matrices are a standard basis of $I_{n-s}^\perp$.
In particular, we may assume that 
$\bar B_i$, for $1 \leq i \leq n-s-1$, is the diagonal matrix whose $i$-th entry is $1$, whose $(n-s)$-th entry is $-1$, and all other entries are $0$.
We fix $B_0$ to be the matrix with a $1$ as entries at $(1,n)$ and $(n,1)$, and all other entries are $0$.
When we choose $b_0 = 1$, $b_i = \varepsilon$ for $1 \leq i \leq n-s-1$, and $b_j = 0$ for $j \geq n-s$, a direct computation reveals that $\bar Z_b$ is the diagonal matrix with entries $(s-n, \ldots, s-n, 1)$.
As $\bar Z_b$ is invertible, we can check that $\bar Z_b$ is not contained in $(I_{n-s}^\perp)^{-1}$.
Hence, $\mathcal{Z} \not\subseteq (I_{n-s}^\perp)^{-1}$, which concludes the proof.
\end{proof}

We now prove the other inclusion $\NM_{k,n}\subseteq\Bad_{k,n}$. Let $\cL$ be a generic point in $\NM_{k,n}$. Then there is an element $Y$ in $\cL^{-1}\cap\cL^{\perp}$, which means that there is a matrix $X\in\cL$ such that $Y$ is contained in $\Gamma_{\cL}(X)$, and this implies $XY=0$. After a change of coordinates we can write \begin{align}
    X = \left[ \begin{array}{c|c}
        I_u & 0  \\ \hline
        0  & 0
    \end{array} \right] \text{ and }
    Y = \left[ \begin{array}{c|c}
        0 & 0  \\ \hline
        0  & I_{v}
    \end{array} \right],
\end{align}
where $u=\rank(X),\; v=\rank(Y)$.  
Now we start by treating some special instances of $\cL$.
We present the proof of the following lemma at the end of this section and first see how it can be applied.

\begin{lemma}\label{lem: not possible}
We have $k>1$. If $(k,u,v) = (3,n-3, 2)$, then $\cL \in  \Bad_{3,n}$.
\end{lemma}

\begin{lemma}\label{lem:u+v not n}
If $u+v<n$, then $\cL$ is contained in $\Bad_{k,n}$.
\end{lemma}
\begin{proof}
Assume that we have $u+v<n$. Note that, since $XY=0$ and $Y$ is an element of $\cL^{\perp}$, we have that $\cL^{\perp}$ is contained in the coisotropic variety Ch$_{c}(D_v)$, where $c=\dim(\cL^{\perp})-\mathrm{codim}(D_v)$. 
By duality, this implies that $\cL$ is contained in Ch$_{k-\mathrm{codim}(D_{n-v})}(D_{n-v})$.
If we have $\binom{v+1}{2}<k$, then $n-v$ is either contained in the Pataki range described in \cite[Theorem 11]{JS} or it exceeds the Pataki range. 
If $n-v$ is in the range, then $\cL$ is contained in $\Bad_{k,n}$ by \cite[Theorem 11]{JS}.
If $n-v$ exceeds the range, then the coisotropic variety Ch$_{k-\mathrm{codim}(D_{n-v})}(D_{n-v})$ is a subvariety of the last coisotropic hypersurface in the Pataki range.
So $\cL$ is again contained in $\Bad_{k,n}$ by \cite[Theorem 11]{JS}, and we are done in both cases. 
Therefore, we can assume that we have \begin{equation}\label{eq:pataki range}\binom{v+1}{2}\geq k.\end{equation} 

Note that the set $\{M\in\SSS^n\mid XM=0\}$ is isomorphic to $\SSS^{n-u}$, and since $\cL^{\perp}$ has codimension $k$, it follows that the space $A_X:=\{M\in\cL^{\perp}\mid XM=0\}$, considered as a subset of $\SSS^{n-u}$, has dimension at least $\binom{n-u+1}{2}-k$. 
Consider the locus $D_{v-1}$ of matrices of rank at most $v-1$ in $\SSS^{n-u}$, which has codimension $\binom{n-u-v+2}{2}$. We conclude that the dimension of $A_X\cap D_{v-1}$ is at least $\binom{n-u+1}{2}-\binom{n-u-v+2}{2}-k$, which by (\ref{eq:pataki range}) is at least  \begin{equation}\label{eq:dimension intersection}\binom{n-u+1}{2}-\binom{n-u-v+2}{2}-\binom{v+1}{2}=(v-1)(n-u-v)-1.\end{equation}
Set $v'=v-1$, and $\delta=n-u-v$. 
If $\dim(A_X \cap D_{v-1}) = 0$, then we have (\ref{eq:dimension intersection}) $<1$, which holds if and only if either $v=1$ or $(v,\delta)=(2,1)$. In the first case, we would have $k=1$ by (\ref{eq:pataki range}). 
In the second case, we have $n-u=3$ and $k = 3$ by (\ref{eq:pataki range}), because otherwise, if $v=2 \geq k$, then we have a strict inequality in (\ref{eq:pataki range}), so $\dim(A_X \cap D_{v-1}) \geq 1$. 
Both of the cases are treated in Lemma \ref{lem: not possible}.

Now assume  $\dim(A_X \cap D_{v-1}) \geq 1$, so we have a matrix in the intersection $A_X\cap D_{v'}$, which, as before, implies that $\cL$ is contained in Ch$_{k-\mathrm{codim}(D_{n-v'})}(D_{n-v'})$. If we have $\binom{v'+1}{2}<k$, then we are again done by \cite[Theorem 11]{JS}. If, on the other hand, we have $\binom{v'+1}{2}\geq k$, then we can do the same computation as before using $v'$ instead of $v$. Therefore, we can continue this until we find an integer $v''<v$ such that $\cL$ is contained in $\mathrm{Ch}_{k-\mathrm{codim}(D_{n-v''})}(D_{n-v''})$, and  $\binom{v''+1}{2}< k$. 
\end{proof}

\begin{prop}\label{prop:NM contained in Bad}
$\NM_{k,n}$ is contained in $\Bad_{k,n}$. 
\end{prop}

\begin{proof}It suffices to show that $\cL$ is contained in $\Bad_{k,n}$. Let $s(\cL)$ be the \textsl{spectrahedral rank} of $\cL$, i.e., $s(\cL)$ is the maximal rank among all the ranks of the positive semidefinite matrices in $\cL$. Define $s(\cL^{\perp})$ analogously. 
For a matrix $M\in\cL$, write $\overline{M}$ for the lower-right $(n-s(\cL))\times(n-s(\cL))$ block of $M$. By \cite[Theorem 5, Corollary 6]{JS}, $\cL$ is \textsl{not} bad if and only if we have both
\begin{equation}\label{eq:being good 1}
    s(\cL)+s(\cL^{\perp})=n,
\end{equation}
and for all $M\in\cL$, we have
\begin{align}\label{eq:being good 2}
&\overline{M}=0 \Rightarrow\mbox{
the upper-right $s(\cL)\times s(\cL^{\perp})$ }\mbox{ block of $M$ is zero. }\end{align}
Assume by contradiction that both (\ref{eq:being good 1}) and (\ref{eq:being good 2}) hold. By Lemma \ref{lem:u+v not n}, we can assume that $u+v=n$, and since $X$ and $Y$ are positive semidefinite, it follows that we have $s(\cL)= u,\;s(\cL^{\perp})= v$.

In what follows, we use the same techniques as in the proof of Lemma~\ref{lem:pointsInLinvernse} and in Example~\ref{ex:accompanyLemma} to show that it follows from our assumptions that $Y$ is not contained in $\Gamma_{\cL}(X)$. From this contradiction we conclude that (\ref{eq:being good 1}) and (\ref{eq:being good 2}) cannot both hold, so $\cL \in \Bad_{k,n}$. 

As before, the elements in $\Gamma_{\cL}(X)$ are given by the adjugates of all perturbations of $X$. 
Let $\{X,B_1,\ldots,B_{k-1}\}$ be a basis for $\cL$. 
For a vector $b=(b_1,\ldots,b_{k-1})\in\mathbb{C}[[\varepsilon]]^{k-1}$, where we write $b_i=b_{i0}+b_{i1}\varepsilon+\cdots$ for all $i$, set $$X_b:=X+\varepsilon\sum_{i=1}^{k-1}b_iB_i.$$
Let $b\in\mathbb{C}[[\varepsilon]]^{k-1}$ be a vector and write $B_b:=\sum_{i=1}^{k-1}b_{i0}\overline{B_i}$. 
As seen before, we have $\overline{\adj(X_b)}=\varepsilon^{n-u-1}Z_b+\mbox{h.o.t}$, with $Z_b:=\adj(B_b)$, and all other entries of $\adj(X_b)$ have higher order terms.

If we have $Z_b\neq0$, then, after dividing by $\varepsilon^{n-u-1}$ and setting $\varepsilon=0$, we have $\overline{\adj(X_b)}=Z_b\in(I_{v}^{\perp})^{-1}$, since $\{\overline{B_1},\ldots,\overline{B_{k-1}}\}$ spans a subspace of $I_v^\perp$. 
Since $\overline{Y}=I_v$, it is not equal to any of the possible non-zero $Z_b$ for varying $b$.

Hence we assume that $b$ is such that $Z_b=0$. If $B_b=0$, we have by~\eqref{eq:being good 2}, $$X_b=\left[ \begin{array}{c|c}
        I_u+\varepsilon(\ldots) & \varepsilon^2(\ldots)  \\
        \hline
        \varepsilon^2(\ldots)  & \varepsilon^2(\ldots)
    \end{array} \right],$$
which contributes the same to $\Gamma_{\cL}(X)$ after a shift of exponent. Hence we assume that $B_b\neq0$. Denote by $c$ the co-rank of $B_b$; since $Z_b=0$, we have $c\geq2$. Let $\varepsilon^d\tilde{Z}_b$ be the lowest degree term of $\adj(X_b)$; we have $d\geq n-u$. 
We will first consider the case when the latter is an equality (in which case $c=2$ as we will see in the following).

The terms in $\adj(X_b)$ of degree $n-u$ are all in the lower-right $(v\times v)$ block. They are constructed in two ways. 
First, take the constant term (which is 1) of an entry in all but one rows of the upper-left $(u\times u)$ block of $X_b$, the linear term of an entry in the remaining row of that block, and the linear term of an entry in every row of $\overline{X}_b$. 
All terms of this shape give $\sum_{r=1}^u\adj(M_{r})$, where $M_r$ is a symmetric $(n-u+1)\times(n-u+1)$ matrix whose first row is given by the $r$-th row of the matrix $\sum_{i=1}^{k-1}b_{i0}B_i$, and the remaining entries are given by $B_b$. 
Second, take the constant term of an entry in \textsl{all} rows of the upper-left $(u\times u)$ block of $X_b$, the linear term of an entry in all but one rows in $\overline{X}_b$, and the quadratic term of an entry in the remaining row.
These terms give the derivative $(d_{B_b}\adj)(\sum_{i=1}^{k-1} b_{i1}\overline{B_i}) $ of the adjugate map at $B_b \in \SSS^{n-u}$ evaluated at the matrix of quadratic terms (in $\varepsilon$).

In particular, we see $c=2$, or else all terms in $\adj(X_b)$ of degree $n-u$ would be zero.
Since $c = 2$, the first row and column of $\adj(M_r)$ is zero.
Moreover, we have $\mathrm{adj}(M_r) \cdot M_r = 0$, which implies that
$\sum_{r=1}^u \overline{\adj(M_r)} \cdot B_b = 0$.
Similarly, we see that 
$(d_{B_b}\adj)(\sum_{i=1}^{k-1} b_{i1}\overline{B_i}) \cdot B_b = 0$.
It follows that $\tilde{Z}_b\cdot B_b=0.$ 
Since $\overline{Y}=I_v,$ it cannot be equal to $\tilde{Z}_b$ since otherwise $B_b$ would be zero. 
This shows that $Y$ can not be the limit of $\adj(X_b)$ for vectors $b$ such that $d = n-u$.

The lowest degree term $\varepsilon^d \tilde{Z}_b$  of $\adj(X_b)$ can be described in a similar fashion if $d > n-u$.
In particular, this term satisfies $\tilde{Z}_b \cdot B_b=0$.
Hence, $Y$ cannot be equal to any $\tilde{Z}_b$ and thus it is not contained in $\Gamma_\cL(X)$; a contradiction.
\end{proof}

Now we can show Theorem~\ref{thm:nonMaximalEqualsBad}, and conclude  with
a proof for Lemma \ref{lem: not possible}. 

\begin{proof}[Proof of Theorem~\ref{thm:nonMaximalEqualsBad}]
This follows from Propositions \ref{prop:badContainedInMLdefective} and \ref{prop:NM contained in Bad}.\end{proof}

\begin{proof}[Proof of Lemma \ref{lem: not possible}]
If $k=1$, then $\cL$ is a point in $\PP\SSS^n$ of full rank. Therefore its ML-degree is 1, and so is the degree of $\PP\cL^{-1}$. This contradicts the fact that $\cL$ is contained in $\NM_{k,n}$.

Now assume $k=3$ and $(u,v) = (n-3,2)$. Assume by contradiction that $\cL$ is not bad, i.e., we assume that (10) and (11) hold. Note that $s(\cL^{\perp})\geq2$, since $Y$ is positive semidefinite. We distinguish two cases: either $(s(\cL),s(\cL^\perp)) = (u,3)$ or $(s(\cL),s(\cL^\perp)) = (u+1,2)$. 

In the first case, there is a positive semidefinite matrix $Y'$ of rank $3$ in $\cL^\perp$ such that $XY'=0$, so $Y'$ is zero away from the lower-right $3 \times 3$ block.
By simultaneous diagonalisation, we may assume $Y$ is as before and $Y'$ is diagonal.
The total transform $\Gamma_\cL(X)$ can be computed following the proof of Proposition~\ref{prop:NM contained in Bad}. 
For $M\in\SSS^n$, we write $\overline{M}$ for the lower-right $3\times3$ block of $M$.
We see that every matrix $Z \in \Gamma_\cL(X)$ satisfies $\overline{Z} \cdot B_b=0$, where $\overline{Z} = \tilde{Z}_b$ as defined in the proof of Proposition~\ref{prop:NM contained in Bad}.
This implies that $\Gamma_\cL(X) \subset \left\{Z\in\SSS^n\mid \overline{Z}\in\left(\overline{Y'}^\perp\right)^{-1}\right\}$, so $Y$ is not in $\Gamma_\cL(X)$; a contradiction.

In the second case, $\cL$ contains a positive semidefinite matrix $B_0$ of rank $u+1$. Since $B_0$ is positive semidefinite, this implies $B_0Y=0$. 
Therefore, the only non-zero entries of $B_0$ are in its upper-left $(u+1)\times(u+1)$ block. 
By simultaneous diagonalisation on this block, we can assume that $X$ is still in the shape it was, and $B_0$ has only entries on the diagonal, and the $(u+1)$-st diagonal entry is 1. Since $\cL$ is regular, it must contain a full-rank matrix, say $B_1$. Now $X,B_0,B_1$ are linearly independent, so they form a basis for $\mathcal{L}$, from which it follows that $B_1$ has full rank in the lower-right $2\times2$ block. 

We want to show that this implies that $Y$ is not contained in $\Gamma_{\cL}(X)$, giving a contradiction. 
As before, we show that $Y$ cannot be constructed by an adjugate of a perturbation of $X$.
For a vector $b=(b_0,b_1)\in\mathbb{C}[[\varepsilon]]^2$, define the matrix $X_b=X+\varepsilon (b_0B_0+b_1B_1)$. 

As before, the lowest degree term in $\overline{\adj(X)}$ is given by $\varepsilon^2Z_b$, with $Z_b=\adj(b_{00}\overline{B_0}+b_{10}\overline{B_1})$.
Since $\overline{B_1}$ has rank 2, the adjugate of $\overline{B_1}$ is non-zero. 
Therefore, either $Z_b\neq0$, or $b_{10}=0$. 
In the latter case, we have $b_1=\varepsilon b_{11}+\varepsilon^2 b_{12}+\cdots$. 
Analogously, The lowest degree non-zero term of $\adj(X_b)$ is of the form $\varepsilon^c \adj\left(b_{00}\overline{B_0}+b_{1,c-2}\overline{B_1}\right)$ for some $c\geq2$. 
Dividing by $\varepsilon^{c}$ and setting $\varepsilon=0$, we see that $\overline{\adj(X_b)}$ cannot be a rank 2 matrix.
Since $\overline{Y}$ has rank 2, we conclude that $Y$ is not in $\Gamma_{\cL}(X)$. 
\end{proof}

\section{Example: the case $n=3$}
In this section we describe in detail the set $\NM_{k,n}$ for $n=3$.

The hypersurface $\NM_{5,3} = \mathrm{Ch}_2(D_1)$ has degree three in $\Gr(5, \SSS^3) \cong \PP^5$, consisting of all subspaces $\cL = A^\perp$ where $\det A = 0$. This matrix satisfies $\tr(A \cdot \adj(A)) = 0$, so $A \in \cL^\perp$. Hence $A \in \cL^{-1} \cap \cL^\perp$.

The hypersurface $\NM_{4,3} \cong \mathrm{Ch}_1(D_1)$ can be identified with $\NM_{2,3} \cong \mathrm{Ch}_1(D_2)$ under $\Gr(4, \SSS^3) \cong \Gr(2, \SSS^3)$. The latter contains those pencils with Segre symbol [2,1] \cite{FMS}, with canonical representation as the span of $A_1 = \left[\begin{smallmatrix}
0 & a & 0 \\
a & 1 & 0 \\
0 & 0 & b
\end{smallmatrix}\right],
A_2 = \left[\begin{smallmatrix}
0 & 1 & 0 \\
1 & 0 & 0 \\
0 & 0 & 1
\end{smallmatrix}\right]$ for $a \neq b \in \mathbb{R}
$.
The matrix $A_1 - a A_2$ looks like
$
\left[\begin{smallmatrix}
0 & 0 & 0 \\
0 & * & 0 \\
0 & 0 & *
\end{smallmatrix}\right]
$
with adjoint
$
\left[\begin{smallmatrix}
* & 0 & 0 \\
0 & 0 & 0 \\
0 & 0 & 0
\end{smallmatrix}\right] \in \cL^{-1} \cap \cL^\perp.
$

The hypersurface $\NM_{3,3} \cong \mathrm{Ch}_2(D_2) \cong \mathrm{Ch}_0(D_1)$. The last term is the Chow form of the Veronese embedding of $\PP^2 \hookleftarrow \PP^5 = \PP(\SSS^3)$, or equivalently, the resultant of three ternary quadrics.

\bigskip
\paragraph{Acknowledgements.}
We thank
Mateusz Micha\l{}ek, Kristian Ranestad, Luca Schaffler, Tim Seynnaeve,  Bernd Sturmfels, and Bernard Teissier for helpful discussions. 
KK was supported by the Knut and Alice Wallenberg Foundation within their WASP (Wallenberg AI, Autonomous Systems and Software Program) AI/Math initiative. 

\bibliography{Q65ref}

\end{document}